\documentclass[reqno,twoside,final]{amsart}
\usepackage{amsmath,amstext,amsthm,amsfonts,amssymb,amscd}

\newcommand{\Image}{{\operatorname{Im}\ }}
\newcommand{\Char}{{\operatorname{Char}\ }}
\newcommand{\Cent}{{\operatorname{Cent}}}
\newcommand{\tr}{{\operatorname{tr}}}
\newcommand{\diag}{{\operatorname{Diag}}}

\newcommand{\GL}{{\operatorname{GL}}}

\newcommand{\N}{\mathbb{N}}

\theoremstyle{definition}

\newtheorem{rem}{Remark}
\newtheorem*{rem*}{Remark}
\newtheorem*{acknow*}{Acknowledgements}
\newtheorem*{examples*}{Examples}
\newtheorem{examples}{Example}
\theoremstyle{plain}

\newtheorem{prm}{Problem}
\newtheorem{lemma}{Lemma}

\newtheorem{prop}{Proposition}
\newtheorem{theorem}{Theorem}
\newtheorem*{theorem*}{Theorem}
\newtheorem{conjecture}{Conjecture}

\newenvironment{proof-sketch}{\noindent{\bf Sketch of Proof}\hspace*{1em}}{\qed\bigskip}
\newenvironment{proof-idea}{\noindent{\bf Proof Idea}\hspace*{1em}}{\qed\bigskip}
\newenvironment{proof-of-lemma}[1]{\noindent{\bf Proof of Lemma #1}\hspace*{1em}}{\qed\bigskip}
\newenvironment{proof-of-prop}[1]{\noindent{\bf Proof of Proposition #1}\hspace*{1em}}{\qed\bigskip}
\newenvironment{proof-of-thm}[1]{\noindent{\bf Proof of Theorem #1.}\hspace*{1em}}{\qed\bigskip}
\newenvironment{proof-attempt}{\noindent{\bf Proof Attempt}\hspace*{1em}}{\qed\bigskip}

\def\Cent{\operatorname{Cent}}
\def\UD{\operatorname{UD}}
\def\sl{\operatorname{sl}}
\def\chara{\operatorname{char}}
\def\a{\alpha}
\def\la{\lambda}

\begin{document}

\title[Images of multilinear polynomials ]{The images of multilinear  polynomials evaluated on $3\times 3$ matrices.}
\author{Alexey Kanel-Belov, Sergey Malev, Louis Rowen}

\address{Department of mathematics, Bar Ilan University,
Ramat Gan, Israel} \email{beloval@math.biu.ac.il}
\email {malevs@math.biu.ac.il}
\email{rowen@math.biu.ac.il}
\thanks{This work was supported
by the  Israel  Science Foundation (grant no.
1207/12)}
\thanks{ The second named author was supported by an Israeli Ministry of Immigrant Absorption
scholarship.}





\begin{abstract}
Let $p$ be a multilinear polynomial in several noncommuting
variables, with coefficients in a algebraically closed field $K$
of arbitrary characteristic. In this paper we classify the
possible images of $p$ evaluated on $3\times 3$ matrices. The
image is one of the following:
\begin{itemize}
\item \{0\}, \item  the set of scalar matrices, \item a (Zariski)
dense subset of $\sl_3(K)$, the matrices of trace 0,
 \item
a dense subset of $M_3(K)$, \item the set of $3-$scalar matrices
(i.e., matrices having eigenvalues $( \beta , \beta \varepsilon,
\beta \varepsilon^2)$ where $\varepsilon$ is a cube root of 1), or
\item the set of scalars plus $3-$scalar matrices.
\end{itemize}
\end{abstract}


\maketitle

\section{Introduction}
This paper is the continuation of \cite{BMR1}, in which we
considered the question, reputedly raised by Kaplansky, of the
possible image set $\Image p$ of a polynomial $p$ on matrices.

\begin{conjecture}\label{Polynomial image}
If $p$ is a multilinear polynomial evaluated on the matrix ring
$M_n(K)$, then $\Image p$ is either $\{0\}$, $K$ (viewed as $K$
the set of scalar matrices), $\sl_n(K)$, or $M_n(K)$.
\end{conjecture} Here $\sl_n(K)$ is the set
of matrices of trace zero.

 This subject was investigated by many
authors (see \cite{AM}, \cite{BK}, \cite{Ch}, \cite{Ku1}, \cite{Ku2},
\cite{LeZh}). For review and basic terminology we refer to our
previous paper \cite{BMR1}. (Connections between images of polynomials on algebras and word equations are
discussed in \cite{BKP}; also see  \cite{La}, \cite{LaS}, \cite{S}.)

 Recall that a polynomial $p$ (written as a sum of monomials) is called {\it
semi-homogeneous of weighted degree $d$} with (integer) {\it
weights} $(w_1,\dots,w_m)$  if for each monomial $h$ of $p$,
taking $d_{j,h}$ to be the degree of $x_{j}$ in $h$, we have
$$d_{1,h}w_1+\dots+d_{n,h}w_n=d.$$ A
semi-homogeneous polynomial with weights $(1,1,\dots, 1)$ is
called $\it{homogeneous}$ of degree $d$.

In  \cite{BMR1} we  settled Conjecture~\ref{Polynomial image} for
$n=2$ and classified the possible images for semi-homogeneous
polynomials:

\begin{theorem}\label{imhom}
Let $p(x_1,\dots,x_m)$ be a semi-homogeneous polynomial
evaluated on the algebra $M_2(K)$ of $2\times 2$ matrices over a
quadratically closed field.
 Then $\Image p$ is either
 $\{0\}$, $K$,  $\sl_2(K)$,  the set of all non-nilpotent matrices in  $\sl_2(K)$,
 or  a
dense subset of $M_2(K)$ (with respect to Zariski topology).
\end{theorem}

 A
 homogeneous polynomial  $p$ is
called $\it{multilinear}$  if $d_{j,h}=1$ for each $1 \le j \le n$
and each monomial $h$ of $p$ (and thus $d = n$).

Examples were given in \cite{BMR1} of homogeneous (but not
multilinear) polynomials whose images do not belong to the
  classification of Theorem 1.

Our research in this paper continues for the $3 \times 3$ case,
yielding the following:
\begin{theorem}\label{main}
If $p$ is a multilinear polynomial evaluated on $3\times 3$
matrices
then $\Image p$ is one of the
following:
\begin{itemize}
\item \{0\},
\item  the set of scalar matrices, \item
 $\sl_3(K)$,  (perhaps lacking the diagonalizable matrices of discriminant
 $0$), cf.~Remark~\ref{rem_sl3}.
 \item
a dense subset of $M_3(K)$, \item the set of $3-$scalar matrices,
or \item  the set of scalars plus $3-$scalar matrices.
\end{itemize}
\end{theorem}

\section{Images of Polynomials}
For any polynomial $p\in K\langle x_1,\dots,x_m\rangle$, the
$image$ of $p$ (in $R$) is defined as
$$\Image p=\{r\in R:
\ \text{there exist}\ a_1,\dots,a_m\in R\ \text{such that}\
p(a_1,\dots,a_m)=r \}.$$

\begin{rem}\label{cong}
$\Image p$ is invariant under conjugation, since $$a
p(x_1,\dots,x_m)a^{-1}=p(a x_1a^{-1},a x_2a^{-1},\dots,a
x_ma^{-1})\in\Image p,$$ for any nonsingular $a\in M_n(K)$.
\end{rem}

We recall the following lemmas (for arbitrary $n$) proved in
\cite{BMR1}:

\begin{lemma}[{\cite[Lemma 4]{BMR1}}]\label{graph}
If $a_i$ are matrix units, then   $p(a_1,\dots,a_m)$ is either~$0$, or~$c\cdot e_{ij}$ for some $i\neq j$, or a diagonal matrix.
\end{lemma}
\begin{lemma}[{\cite[Lemma 5]{BMR1}}]\label{linear}
The linear span of   $\Image p$ is either $\{0\}$, $K$,
$\sl_n$, or $M_n(K)$. If $\Image p$ is not $\{0\}$ or the set of scalar matrices, then
for any $i\neq j$ the matrix unit $e_{ij}$ belongs to $\Image p$.
\end{lemma}

Another major tool is  Amitsur's Theorem \cite[Theorem 3.2.6,
p.~176]{Row}, that the algebra of generic $n\times n$ matrices
(generated by matrices $Y_k = (\xi _{i,j}^{(k)})$ whose entries
$\{ \xi _{i,j}^{(k)}, 1 \le i,j \le n\}$ are commuting
indeterminates) is a non-commutative domain $\UD$ whose ring of
fractions with respect to the center is a division algebra which
we denote as $\widetilde{\UD}$ of dimension $n^2$ over its center
$F_1: = \Cent(\widetilde{\UD}))$.

\begin{rem}\label{patch1} Suppose $t$ is a commuting indeterminate, and $f(x_1,
\dots, x_m;t)$ is a polynomial taking values under matrix
substitutions for the~$x_i$ and scalars for $t.$  Suppose that there exists
unique $t_0$  such that $f(x_1, \dots, x_m;t_0)=0$.

In case  $\Char (K)=0$.    $t_0$ is
a rational function with respect to the entries of $x_i$.
If this
$t_0$ is fixed under  simultaneous conjugation of generic matrices
$x_1,\dots ,  x_m$, then $t_0$ is in the center of Amitsur's
division algebra ~$\widetilde{\UD}$, implying $f \in
\widetilde{\UD}$.

In case $\Char(K)=q\neq 0$, then $t_0^{q^l}$ is a
rational function for some $l\in\mathbb{N}_0$.
\end{rem}
\begin{rem}
In Remark \ref{patch1} we could take a system of polynomial
equations and polynomial inequalities. If $t_0$ is unique, then it
is a rational function (or $t_0^{q^l}$ if $\Char(K)=q>0$).
\end{rem}

In fact, we need a slight modification of Amitsur's theorem, which
is well known. Viewing    $$\widetilde{\UD} \subseteq M_n\left(
F(\xi _{i,j}^{(k)}): 1\le i,j \le n,\ k\ge 1\right)$$ we can
define the reduced characteristic coefficients of elements of
$\widetilde{\UD}$, which by \cite[Remark~24.67]{Row2} lie in
$F_1$.

\begin{lemma}\label{divAm}  Suppose that an element $a$ of $\widetilde{\UD}$ has
a unique eigenvalue~$\a$ (i.e., of multiplicity $n$). If $\Char(K)=0$,  then $a$ is
scalar. If $\Char(K)=q>0$, then $q|n$ and $a$ is $q^l-$scalar for some
$l$.
\end{lemma}
\begin{proof} If $\Char(K)=0,$ then $\a$ is the element of $\widetilde{\UD}$ and $a -\a I$ is nilpotent, and thus
$0.$

If $\Char(K)=q$ then $\a^{q^l}$ is an element of
$\widetilde{\UD}$; therefore $a^{q^l} -\a^{q^l} I$ is nilpotent,
and thus $0,$ implying $a$ is $q^l-$scalar. This is impossible if $q$
does not divide the size of the matrices $n$.
\end{proof}

\begin{lemma}\label{div2} The multiplicity of any eigenvalue of an element $a$ of
$\widetilde{\UD}$ must divide~$n$. In particular, when $n$ is odd,
$a$ cannot have an eigenvalue  of multiplicity 2.
\end{lemma}
\begin{proof} Recall \cite[Remark~4.106]{Row1}
 that for any element $a$ in a division
algebra, represented as a matrix, the eigenvalues of $a$ occur
with the same multiplicity, which thus must divide $n$.
\end{proof}

\begin{prop}\label{Am10} Suppose we have a homomorphism $\varphi:\widetilde{\UD} \to A$
given by the specialization $\varphi(Y_k) = a_k.$ Then any
characteristic coefficient of $ Y_k $ in $\widetilde{\UD}$
specializes to the corresponding characteristic coefficient of
$a_k.$
\end{prop}
\begin{proof} Let $t:=n^2$. Any characteristic coefficient of an element of $\widetilde{\UD}$ can be expressed as the ratio of two
central polynomials, in view of \cite[Theorem 1.4.12]{Row}; also
see \cite[Theorem~J, p.~27]{BR} which says that  for any
$t-$alternating polynomial nonidentity $f$, and for any
characteristic coefficient $\omega_\ell  $ of the characteristic
polynomial  $\lambda^t + \sum_{\ell =1}^t (-1)^\ell  \omega _\ell
\lambda ^{t-\ell }$ of a linear transformation $T$ of the
$t$-dimensional vector space corresponding to $n \times n$
matrices,  we have
\begin{equation}\label{trace2pol0}
\omega_\ell(T)  f(a_1, \dots, a_t, r_1, \dots, r_m) = \sum
 f(T^{\ell
_1}a_1, \dots, T^{\ell _t} a_t,  r_1, \dots, r_m) ,
\end{equation}
summed over all vectors $(\ell _1, \dots, \ell _t)$ where each
$\ell _i \in \{ 0, 1 \}$ and $\sum \ell _i = l.$  Hence, taking $ f(a_1, \dots, a_t, r_1, \dots, r_m)\ne
0$, the
characteristic coefficient of a polynomial evaluated on $A$ is obtained according
to the specialization from $\widetilde{\UD}$ induced from~
$\varphi$.
\end{proof}

We recall Donkin's theorem:
\begin{theorem}[Donkin~\cite{D}]\label{Donkin}
For any $m, n\in \N$, the algebra of polynomial invariants
$K[M_n(K)^m]^{\GL_n(K)}$ under $\GL_n(K)$  is generated by the
trace functions
\begin{equation}\label{Donk} T_{i,j}(x_1,x_2,\dots,x_m) =
\operatorname{Trace}(x_{i_1}x_{i_2}\cdots
x_{i_r},\bigwedge\nolimits^jK^n),\end{equation} where
$i=(i_1,\dots,i_r),$ all $i_l\leq m,$ $r\in\N, j>0,$ and
$x_{i_1}x_{i_2}\cdots x_{i_r}$ act  as linear transformations on
the exterior algebra $\bigwedge^jK^n$.
\end{theorem}

Proposition~\ref{Am10} yields the following observation:

\begin{prop}\label{Am1} All of Donkin's  invariants of  Theorem \ref{Donkin}   can be embedded
in ~$\widetilde{\UD}$.
\end{prop}

For $n>2$, we also have an easy consequence of the theory of
division algebras.

\begin{lemma}\label{div}  Suppose for some polynomial $p$ and some  number $q< n$,
 that $p^q$ takes on only scalar values in $M_n(K)$, over an infinite field $K$, for $n$ prime. Then  $p$
 takes on only scalar values in $M_n(K)$.
\end{lemma}
\begin{proof}  We can view $p$ as an element of the generic
division algebra $\widetilde{\UD}$ of degree~$n$, and we adjoin a
$q$-root of 1 to $K$ if necessary. Then $p$ generates a subfield
of dimension 1 or $n$ of $\widetilde{\UD}$. The latter is
impossible, so the dimension is 1; i.e., $p$ is already central.
\end{proof}

\subsection{The case $M_3(K)$}$ $

Now we turn specifically to  the case $n=3$.
 Let $K$ be an  algebraically closed field.
We say that a polynomial $p$ is {\bf trace-vanishing} if each of
its evaluations have trace 0; i.e., $\tr(p)$ is a trace identity
of $p$. Also, for $\chara(K) \ne 3$ we fix a  primitive cube root
$\varepsilon \ne 1$  of $1$; when $\chara(K) = 3$ we take
$\varepsilon = 1$.

\begin{lemma}\label{extr}
We define functions $\omega _k : M_3(K) \to K$ as follows: Given a
matrix $a$, let $\lambda_1,\lambda_2, \lambda _3$ be the
eigenvalues of $a,$ and denote
$$\omega_k : = \omega_k (a) =\sum\limits_{1\leq i_1<i_2<\dots<i_k\leq
3}\lambda_{i_1}\dots \lambda_{i_k}.$$ Let $p(x_1,\dots,x_m)$ be
a semi-homogeneous, trace-vanishing polynomial.

 Consider the rational function
$H(x_1,\dots,x_m)
=\frac{\omega_2(p(x_1,\dots,x_m))^3}{\omega_3(p(x_1,\dots,x_m))^2}$
(taking values in $K \cup \{\infty\}$). If $\Image H$ is dense in
~$K$, then  $\Image p$ is dense in $\sl_3$.
\end{lemma}
\begin{proof}  Note that $\omega_2(p)^3$ and $\omega_3(p)^2$ are
semi-homogeneous. Thus, $\Image H$ is dense in $K$ iff the image
of the pair $(\omega_2(p)^3,\omega_3(p)^2)$ is dense in~$K^2$. But
since $\omega_2$ and $\omega_3$ are algebraically independent, so
are $\omega_2(p)^3$ and $\omega_3(p)^2$, so we conclude that the
image of the pair $(\omega_2(p)^3,\omega_3(p)^2)$ is dense in
$K^2$. Thus, the set of characteristic polynomials of evaluations
of $p$ is dense in the space of all possible characteristic
polynomials of trace zero matrices. Therefore, the set of all
triples $(\lambda_1,\lambda_2,-\lambda_1-\lambda_2)$ of
eigenvalues of matrices from $\Image p$ is dense in the plane
$x+y+z=0$ defined in $K^3$, implying that $\Image p$ is dense in
$\sl_3$.
\end{proof}

\begin{theorem}\label{semi_tr0_3}
Let $p(x_1, \dots, x_m)$ be a semi-homogeneous polynomial which is
trace-vanishing on $3\times 3$ matrices. Then $\Image p$ is one of
the following:

\begin{itemize}
\item \{0\}, \item  the set of scalar matrices (which can occur
only if $\Char K=3$),
 \item a dense subset of  $\sl_3(K)$, or \item
the set of $3-$scalar matrices, i.e., the set of matrices with
eigenvalues $(\gamma,\gamma\varepsilon,\gamma\varepsilon^2)$,
where $\varepsilon$ is our cube root of $1$.
\end{itemize}
\end{theorem}

\begin{proof-of-thm} {\ref{semi_tr0_3}}
We define the functions $\omega _k : M_n(K) \to K$ as in
Lemma~\ref{extr}, and consider the rational function
$H=\frac{\omega_2(p(x_1,\dots,x_m))^3}{\omega_3(p(x_1,\dots,x_m))^2}$
(taking values in $K \cup \{\infty\}$).

If $\omega_2(p)=\omega_3(p)=0$,  then each evaluation of $p$ is a
nilpotent matrix, contradicting Amitsur's Theorem. Thus, either
$\Image H$  is dense in $K$, or $H$ must be constant.

If $\Image H$ is dense in $K$, then  $\Image p$ is dense in $\sl_3$ by Lemma~\ref{extr}.

So we may assume that $H$ is a constant, i.e.,
 $\alpha \omega_2^3(p)+\beta \omega_3^2(p)=0$ for
some  $\alpha ,\beta \in K$ not both $0$. Fix generic matrices
$Y_1,\dots,Y_m$. We claim that the eigenvalues
$\lambda_1,\lambda_2,-\lambda_1-\lambda_2$ of
$q:=p(Y_1,\dots,Y_m)$
  are pairwise distinct. Otherwise either they  are all
equal, or two of them are equal and the third is not, each of
which is impossible by Lemmas~\ref{divAm} and \ref{div2} since $q
\in \widetilde{\UD}$.

 Let
$\lambda_1',\lambda_2',-\lambda_1'-\lambda_2'$ be the eigenvalues of
another matrix $r\in\Image p.$ Thus we have the following:
$$\alpha\omega_2^3(r)+\beta\omega_3^2(r)=\alpha\omega_2^3(q)+\beta\omega_3^2(q)=0.$$
Therefore we have homogeneous equations on the eigenvalues.
Dividing by $\lambda_2^6$ and $\lambda_2'^6$ respectively, we have
the same two polynomial equations of degree 6 on
$\frac{\lambda_1}{\lambda_2}$ and $\frac{\lambda'_1}{\lambda'_2}$,
yielding six possibilities for $\frac{\lambda'_1}{\lambda'_2}$.
The six permutations of $\lambda_1,\lambda_2,$ and
$\lambda_3=-\lambda_1-\lambda_2$  define six pairwise different
$\frac{\lambda'_1}{\lambda'_2}$ unless
$(\lambda_1,\lambda_2,\lambda_3)$ is   a permutation (multiplied
by a scalar) of one of the following triples: $(1,1,-2),\
(1,-1,0),(1,\varepsilon,\varepsilon^2).$ The first case is
impossible since the eigenvalues must be pairwise distinct. The
second case give us an element of Amitsur's algebra
$\widetilde{\UD}$ with eigenvalue $0$ and thus determinant 0,
contradicting Amitsur's Theorem. In the third case the polynomial
$p$ is $3-$scalar. Thus, either $p$ is $3-$scalar polynomial, or
each matrix from $\Image p$ will have the same eigenvalues up to
permutation and scalar multiple. Note for $p$ being $3-$scalar
this is true also.

Assume that for some $i\in\{2,3\}$ that $\tr( p^i)$ is not
identically zero. Then $\lambda_1^i, \lambda_2^i$, and
$\lambda_3^i$ are three linear functions on $\tr(p^i)$. Hence we
have the PI (polynomial identity)
$(p^i-\lambda_1^i)(p^i-\lambda_2^i)(p^i-\lambda_3^i)$. Thus by
Amitsur's Theorem, one of the factors is a PI. Hence $p^i$ is a
scalar matrix. However $i\neq 2$ by Lemma~\ref{div2}. Hence $i=3.$
In this case the image of $p$ is the set of matrices with
eigenvalues $\{(\gamma,\gamma\varepsilon,\gamma\varepsilon^2):
\gamma \in K\}$.

Thus, we may assume that $p$ satisfies $\tr(p^i)=0$ for $i=1,\ 2$
and $3$.  Now $\omega_1(p)=\tr(p)=0$ and
$2\omega_2(p)=(\tr(p))^2-\tr(p^2)=0$.

Hence $\omega_1=\omega_2=0$ if $\chara(K) \ne 2$; in this case
$\omega_3$ is either $0$ (and hence $p$ is PI) or not $0$ (and
hence $p$ is $3-$scalar).

So assume that $\chara (K)=2$. Recall that
$$0=\tr
(p^3)=\lambda_1^3+\lambda_2^3+\lambda_3^3=\lambda_1^3+\lambda_2^3+\lambda_3^3-3\lambda_1\lambda_2\lambda_3+3\lambda_1\lambda_2\lambda_3.$$
But
$\lambda_1^3+\lambda_2^3+\lambda_3^3-3\lambda_1\lambda_2\lambda_3$
is a multiple of $\lambda_1+\lambda_2+\lambda_3$ (seen by
substituting $-(\lambda_1+\lambda_2)$ for $\lambda_3)$ and thus
equals $0$. Thus, $0 = 3 \lambda_1\lambda_2\lambda_3 =
\lambda_1\lambda_2\lambda_3 =\omega_3(p)$, and the Hamilton-Cayley
equation yields $p^3+\omega_2p=0.$ Therefore, $p(p^2+\omega_2)=0$
and by Amitsur's Theorem either $p$ is PI, or $p^2 = -\omega_2$
(which is central), implying by Lemma~\ref{div} that $p$ is
central.
\end{proof-of-thm}

\begin{examples}
The   element  $ [x,[y,x]x[y,x]^{-1}]$ of $\widetilde{UD}$ takes on only   $3-$scalar values (see
\cite[Theorem 3.2.21, p.~180]{Row}) and thus gives rise to a homogeneous polynomial
taking on  only $3-$scalar values.

\end{examples}

Now we consider the possible image sets of multilinear
trace-vanishing polynomials.

\begin{lemma}\label{diag-not-scal}
If $p$ is a multilinear polynomial, not PI nor central, then there
exist a collection of matrix units $(E_1,E_2,\dots,E_m)$ such that
$p(E_1,E_2,\dots,E_m)$ is a diagonal but not scalar matrix.
\end{lemma}
\begin{proof}
By Lemmas \ref{graph} and \ref{linear}, the linear span of all
$p(E_1,E_2,\dots,E_m)$ for any matrix units $E_i$ such that
$p(E_1,E_2,\dots,E_m)$ is diagonal includes all
$\diag\{x,y,-x-y\}$. In particular there exist a collection of
matrix units $(E_1,E_2,\dots,E_m)$ such that
$p(E_1,E_2,\dots,E_m)$ is a diagonal but not scalar matrix.
\end{proof}

\begin{theorem}\label{multi_tr=0_3}
Let $p$ be a   multilinear  polynomial which is trace-vanishing on
$3\times 3$ matrices  over a field $K$ of arbitrary
characteristic. Then $\Image p$ is one of the following:

\begin{itemize}
\item \{0\}, \item  the set of scalar matrices,
\item the set of $3-$scalar matrices, or
\item  for each triple
$\lambda_1+\lambda_2+\lambda_3=0$ there exist a matrix $M\in\Image
p$ with eigenvalues $\lambda_1,\ \lambda_2$ and $\lambda_3$.
\end{itemize}

\end{theorem}

\begin{proof} If the polynomial $\omega_2 (p)$ (defined in the proof of
Theorem~\ref{semi_tr0_3}) is identically zero, then the
characteristic polynomial is $p^3-\omega_3(p)=0$, implying $p$ is
either scalar (which can happen  only if $\Char (K)=3$) or
$3-$scalar. Therefore we may assume that the polynomial $\omega_2
(p)$ is not identically zero. Let
$$f_{\alpha ,\beta }(M)=\alpha \omega_2(M)^3+\beta \omega_3(M)^2.$$  It is enough to
show that for any $  \alpha , \beta \in K$ there exists a
non-nilpotent matrix $M= p(a_1,\dots,a_m)$ such that $f_{\alpha
,\beta }(p(a_1,\dots,a_m))=0,$ since this will imply that the
image of $H$ (defined in Lemma~\ref{extr}) contains all $-\frac {\beta}{\alpha}$ and thus  $K
\cup \{\infty\}$. (For example, if $\alpha= 0$ and $\beta \ne 0,$ then $ \omega _3(M) = 0,$ implying  $ \omega _2(M) \ne 0$
since $ \omega _1(M) = 0 $ and $M$~is non-nilpotent, and thus $H = \infty.)$ Therefore,
for any trace-vanishing  polynomial (i.e., a polynomial
$x^3+\gamma _1 x+\gamma _0$) there is a matrix in $\Image p$ for
which this is the characteristic polynomial. Hence whenever
$\lambda_1+\lambda_2+\lambda_3=0$ there is a matrix with
eigenvalues $\lambda_i$.

Without loss of generality we may assume that $a=p(Y_1,\dots,Y_m)$
and $b=p(\tilde Y_1,Y_2\dots,Y_m)$ are not proportional for
generic matrices $\tilde Y_1,Y_1,\dots,Y_m$,
cf.~\cite[Lemma~2]{BMR2}. Consider the polynomial $\varphi_{\alpha
,\beta }(t)=f_{\alpha ,\beta }(a+tb)$. There are three cases to
consider:

CASE I. $\varphi_{\alpha ,\beta }= 0$ identically. Then $f_{\alpha
,\beta }(a)=0$, and $a$ is not nilpotent by Proposition~\ref{Am1}.

CASE II. $\varphi_{\alpha ,\beta }$ is a constant
 $\tilde \beta  \ne 0.$  Then
$f_{\alpha ,\beta }(b+ta)=t^6\varphi_{\alpha ,\beta }(t^{-1})=\tilde
\beta t^6$; thus $f_{\alpha ,\beta }(b)=0$, and $b$ is not
nilpotent by Proposition~\ref{Am1}.

CASE III. $\varphi_{\alpha ,\beta }$ is not  constant.  Then it
has finitely many roots. Assume that for each substitution $t$ the
matrix $a+tb$ is nilpotent; in particular, $\omega_2(a+tb)=0$.
Note that $\omega_2(a+tb)$ equals  the sum of principal $2\times
2$ minors and thus is a quadratic polynomial (for otherwise
$\omega_2(b)=0$ which means that $\omega_2(p)$ is identically
zero, a contradiction). Hence $\omega_2(a+tb)$  has two roots,
which we denote as $t_1$ and $t_2$. If $t_1=t_2$, then $t_1$ is
 uniquely defined and thus, in view of Remark~\ref{patch1}, is a rational function in the entries of $a$ and $b$, and $a+t_1b$ is
a nilpotent rational function (because we assumed that one of
$a+t_1b$ and $a+t_2b$ is nilpotent, but here they are equal.) At
least one of $t_1$ and $t_2$ is a root of $\varphi_{\alpha ,\beta
}$.

If only $t_1$ is a root, then $t_1$ is uniquely defined and thus,
by Remark~\ref{patch1}, is a rational function; hence,  $a+t_1b$
is a nilpotent polynomial, contradicting Proposition~\ref{Am1}.
Thus, we may assume that both $t_1$ and $t_2$ are roots of
$\varphi_{\alpha ,\beta }$. But $\varphi_{\alpha ,\beta }(t_i)$ is
nilpotent, and in particular $\omega_3(a+t_ib)=0$. Thus there
exists exactly one more root $t_3$ of $\omega_3(a+tb)$, which is
uniquely defined and thus, by Remark~\ref{patch1}, is rational.
Hence we may consider the polynomial $q(x_1,\dots,x_m,\tilde
x_1)=a+t_3b$, which must satisfy the condition $\tr (q)=\det
(q)=0$.  This is impossible for homogeneous $q$ by Theorem
\ref{semi_tr0_3}, and also impossible for nonhomogeneous $q$ since
the leading homogenous component $q_d$ would satisfy $\tr
(q_d)=\det (q_d)=0$, a contradiction.
\end{proof}
\begin{rem}
Assume that $\chara (K) = 3$ and $p$ is a multilinear polynomial,
which is neither PI nor central. Then, according to Lemma \ref{diag-not-scal}
there exists a collection of matrix units $E_i$ such that
$$p(E_1,\dots,E_m)=\diag\{\alpha ,\beta ,\gamma\}$$ is diagonal but not
scalar. Without loss of generality,  $\alpha \neq \beta $. Hence
$p^3(E_1,\dots,E_m)=\diag\{a^3,\beta ^3,\gamma^3\}$ and $\alpha
^3\neq \beta ^3$ because $\chara (K) = 3$. Therefore $p$ is not
$3-$scalar.
\end{rem}
\begin{theorem}\label{dense}
If there exist $\alpha ,\ \beta ,$ and $\gamma$ in $K$ such that
$\alpha +\beta +\gamma,\ \alpha +\beta
\varepsilon+\gamma\varepsilon^2$ and $\alpha +\beta
\varepsilon^2+\gamma\varepsilon$ are nonzero, together with matrix
units $E_1,E_2,\dots,E_m$ such that $p(E_1,E_2,\dots,E_m)$ has
eigenvalues $\alpha ,\ \beta $ and $\gamma$, then $\Image p$ is
dense in $M_3$.
\end{theorem}
\begin{proof}
Define $\chi$ to be the permutation  of the set of matrix units,
sending the indices $1\to 2$, $2\to 3$, and $3\to 1$. For example,
$\chi(e_{12})=e_{23}$.
For triples $T_1,\dots,T_m$ (each $T_i=(t_{i,1},t_{i,2},t_{i,3})$)
consider the function

\begin{eqnarray}\label{mapping_triples}
f(T_1,\dots,T_m)=p(t_{1,1}x_1+t_{1,2}\chi(x_1)+t_{1,3}\chi^{-1}(x_1),t_{2,1}x_2+t_{2,2}\chi(x_2)
+t_{2,3}\chi^{-1}(x_2),\cr
\dots,t_{m,1}x_m +t_{m,2}\chi(x_m)+t_{m,3}\chi^{-1}(x_m)).
\end{eqnarray}

Opening the brackets, we   have $3^m$ terms, each of which we
claim is   a diagonal matrix. Each term is a monomial with
coefficient of the type
$$\chi^{k_{\pi(1)}} \chi^{k_{\pi(2)}}\cdots \chi^{k_{\pi(m)}}x_{\pi(1)}x_{\pi(2)}\cdots x_{\pi(m)},$$
where $k_i$ is $-1,0$ or $1$, and $\pi$ is a permutation. Since we
substitute only matrix units in $p$, by Lemma \ref{graph} the
image is either diagonal or a matrix unit with some coefficient.
For each of the three vertices $v_1,v_2,v_3$ in our graph define
the index $\iota _\ell $, for $1 \le \ell \le 3$ to be   the
number of incoming edges to $v_\ell$ minus the number of outgoing
edges from $v_\ell$. Thus, at the outset, when the image is
diagonal, we have $\iota _1 = \iota _2 = \iota _3 = 0.$

We claim that after applying $\chi$ to any matrix unit the new
$\iota'_\ell$ will all still be congruent modulo 3.
Indeed, if the edge $\vec{12}$ is changed to $\vec{23}$, then
$\iota'_1 = \iota +1$ and $\iota_3' = \iota_3+1,$ whereas
$\iota'_2 = \iota_2 -2 \equiv \iota_2 +1.$
 The same with changing
$\vec{23}$ to $\vec{31}$ and $\vec{31}$ to $\vec{12}$. If we make
the opposite change $\vec{21}$ to $\vec{13},$ then (modulo 3) we
subtract $1$ throughout. If we make a change of the type
$\vec{ii}\mapsto \vec{jj}$, then $\iota_\ell ' = \iota_\ell $ for
each $\ell$.

If $p(\chi^{k_1}x_1,\chi^{k_2}x_2,\dots,\chi^{k_m}x_m)=e_{ij}$,
this means that the number of incoming edges minus the number of
outgoing edges   of the vertex $i$ is $-1 \pmod 3$ and the number
of incoming edges minus the number of outgoing edges   of $j$ is
$1\pmod 3$, which are not congruent modulo $3$. Thus the values of
the mapping $f$ defined in \eqref{mapping_triples} are diagonal
matrices. Now fix $3m$ algebraically independent triples
$T_1,\dots,T_m,\Theta_1,\dots,\Theta_m,\Upsilon_1,\dots,\Upsilon_m.$
Assume that $\Image f$ is $2-$dimensional. Then $\Image df$ must
also be  $2-$dimensional at any point. Consider the differential
$df$ at the point $(\Theta_1,T_2,\dots,T_m)$. Thus,
$$f(\Theta_1,T_2,\dots,T_m),\ f(T_1,T_2,\dots,T_m),\
f(\Theta_1,\Theta_2,\dots,T_m)$$ belong to $\Image\ df$. Thus
these three matrices must span a linear space of dimension not
more than $2$. Hence they lie in some plane $P$. Now take
$$f(\Theta_1,\Theta_2,T_3,\dots,T_m),\
f(\Theta_1,T_2,T_3,\dots,T_m),\
f(\Theta_1,\Theta_2,\Theta_3,T_4,\dots,T_m).$$ For the same reason
they lie in a plane, which is
  the   plane $P$ because it has
two vectors from $P$. By the same argument, we conclude that all
the matrices of the type
$f(\Theta_1,\dots,\Theta_k,T_{k+1},\dots,T_m)$ lie in P. Now we
see that $$f(\Theta_1,\dots,\Theta_{m-1},T_m),\
f(\Theta_1,\dots,\Theta_m), \
f(\Upsilon_1,\Theta_2,\dots,\Theta_m)$$ also lie in $P$.
Analogously we obtain that also
$$f(\Upsilon_1,\dots,\Upsilon_k,\Theta_{k+1},\dots,\Theta_m)\in
P$$ for any $k$.

 Hence for $3m$ algebraically independent triples
$$T_1,\dots,T_m;\Theta_1,\dots,\Theta_m;\Upsilon_1,\dots,\Upsilon_m,$$
we have obtained that $f(T_1,\dots,T_M)$,
$f(\Theta_1,\dots,\Theta_m)$ and $f(\Upsilon_1,\dots,\Upsilon_m)$
lie in one plane. Thus any three values of $f$, in particular
$\diag\{\alpha ,\beta ,\gamma\},\ \diag\{\beta ,\gamma,\alpha \}$
and $\diag\{\gamma,\alpha ,\beta \},$ must lie in one plane. We
claim that this can happen only if $$\alpha +\beta +\gamma=0,\quad
\alpha +\beta \varepsilon+\gamma\varepsilon^2=0,\quad
\text{or}\quad \alpha +\beta \varepsilon^2+\gamma\varepsilon=0.$$

Indeed, $\diag\{\alpha ,\beta ,\gamma\},\ \diag\{\beta
,\gamma,\alpha \}$ and $\diag\{\gamma,\alpha ,\beta \},$ are
dependent if and only if the matrix
$$ \left( \begin{matrix} \alpha \ \beta \ \gamma \\ \beta \ \gamma\ \alpha  \\ \gamma\ \alpha \ \beta   \end{matrix}\right)$$
is singular, i.e., its determinant $3\alpha \beta \gamma -(\alpha
^3+\beta ^3+\gamma^3) = 0$. But this has the desired three roots
when viewed as a cubic equation in $\gamma$.

We have a contradiction to our hypothesis.
\end{proof}

\begin{rem}\label{1-codim}
If there exist $\alpha ,\ \beta ,$ and $\gamma$ such that $\alpha
+\beta +\gamma=0$ but $(\alpha ,\beta ,\gamma)$ is not
proportional to $(1,\varepsilon,\varepsilon^2)$ or
$(1,\varepsilon^2,\varepsilon)$,
 with
matrices $ E_1,E_2,\dots,E_m $ such that $p(E_1,E_2,\dots,E_m)$
has eigenvalues $\alpha ,\ \beta $ and $\gamma$, then either all
diagonalizable trace zero matrices lie in $\Image p$, or $\Image
p$ is dense in $M_3(K)$. If $\alpha +\beta
\varepsilon+\gamma\varepsilon^2=0$ but $(\alpha ,\beta ,\gamma)$
is not proportional to $(1,\varepsilon,\varepsilon^2)$ or
$(1,1,1)$, then all diagonalizable matrices with eigenvalues
$\alpha +\beta ,\ \alpha +\beta \varepsilon$ and $\alpha +\beta
\varepsilon^2$ lie in $\Image p$ or $\Image p$ is dense in
$M_3(K)$.
\end{rem}
\begin{rem}\label{char3-classification}
The proof of Theorem \ref{dense} works also for any field $K$ of
characteristic $3$. In this case $\varepsilon=1$. Hence, if there
are $\alpha ,\ \beta ,$ and $\gamma$ in $K$ such that $$\alpha
+\beta +\gamma\ne 0,$$ together with matrix units
$E_1,E_2,\dots,E_m$ such that $p(E_1,E_2,\dots,E_m)$ has
eigenvalues $\alpha ,\ \beta $ and $\gamma$, then $\Image p$ is
dense in $M_3$. Therefore, for $\Char K=3$, any multilinear
polynomial $p$ is either trace-vanishing or $\Image p$ is dense in
$M_3(K).$
\end{rem}

\begin{theorem}\label{equation}
If p is a multilinear polynomial such that $\Image p$ does not
satisfy the equation $\gamma\omega_1(p)^2=\omega_2(p)$ for
$\gamma=0$ or $\gamma=\frac{1}{4}$, then $\Image p$ contains a
matrix with two equal eigenvalues that is not diagonalizable and
of determinant not zero. If $\Image p$ does not satisfy any
equation of the form $\gamma\omega_1(p)^2=\omega_2(p)$ for any
$\gamma$, then the set of non-diagonalizable matrices of $\Image
p$ is Zariski dense in the set of all non-diagonalizable matrices,
and $\Image p$ is dense.
\end{theorem}
\begin{proof}
If not, then  by \cite[Lemma~2]{BMR2} there is at least one
variable (say, $x_1$) such that $a = p(x_1,x_2,\dots,x_m)$ does
not commute with $b = p(\tilde x_1,x_2,\dots,x_m)$. Consider the
matrix
 $a+tb=p(x_1+t\tilde x_1,x_2,\dots,x_m)$, viewed as a polynomial in $t$.

Recall that the discriminant of a $3\times 3$ matrix with
eigenvalues $\la _1, \la_2, \la _3$ is defined as $\prod _{1 \le i
< j \le 3}(\la_i - \la _j)^2$. Thus, the discriminant of $a+tb$ is
a polynomial $f(t)$ of degree~$6$.
 If $f(t)$ has only one root $t_0$, then this
root is  defined in terms of the entries of $\tilde
x_1,x_1,x_2,\dots,x_m$, and invariant under the action of the
symmetric group,  and thus is in Amitsur's division algebra
$\widetilde{\UD}$. By Lemma~\ref{divAm}, $a+t_0b$ is scalar, and
the uniqueness of $t_0$ implies that $a$ and $b$ are scalar,
contrary to assumption.

Thus, $f(t)$ has at
least two roots - say, $t_1\neq t_2$, and
 the matrices $a+t_1b$ and $a+t_2b$ each must have multiple
eigenvalues. If both of these matrices are diagonalizable, then
each of $a+t_ib$ have a $2-$dimensional plane of eigenvectors.
Therefore we have two $2-$dimensional planes in $3-$dimensional
linear space, which must intersect. Hence there is a common
eigenvector of both $a+t_ib$ and this is a common eigenvector of
$a$ and $b$. If $a$ and $b$ have a common eigenspace of dimension
1 or 2, then there is at least one eigenvector (and thus
eigenvalue) of $a$ that is uniquely defined, implying $a \in
\widetilde{UD}$ by Remark~\ref{patch1}, contradicting
Lemma~\ref{divAm}. If $a$ and $b$ have a common eigenspace of
dimension 3, then $a$ and $b$ commute, a contradiction.

We claim that there cannot be a diagonalizable matrix with equal
eigenvalues on the line $a+tb$. Indeed, if there were   such a
matrix, then either it would be unique (and thus an element of
$\widetilde {UD}$, which cannot happen), or there would be at least two such
diagonalizable matrices, which also cannot happen, as shown above.

Assume  that all   matrices on the line $a+tb$ of discriminant
zero have determinant zero. Then either all of them are of the
type $\diag\{\lambda,\lambda,0\}+e_{12}$ or all of them are of the
type $\diag\{0,0,\mu\}+e_{12}$. (Indeed, there are three roots of
the determinant equation $\det(a+tb)=0$, which are pairwise
distinct, and all of them give a matrix with two equal
eigenvalues,   all belonging to one of these types, since
otherwise one eigenvalue is uniquely defined and thus yields an
element of $\widetilde {UD}$, which cannot happen.

In the first case, all three roots of the determinant equation
$\det(a+tb)=0$ satisfy  the equation
$(\omega_1(a+tb))^2=4\omega_2(a+tb)$. Hence, we have three
pairwise distinct roots of the polynomial of maximal degree $2$,
which can occur only if the polynomial is identically zero. It
follows that   also $(\omega_1(a))^2-4\omega_2(a)=0$,
 so $(\omega_1(p))^2-4\omega_2(p)=0$ is
identically zero, which by hypothesis cannot happen.

 In the second case we have
 the analogous situation, but $\omega_2(p)$ will be
identically zero, a  contradiction.

Thus on the line $a+tb$ we have at least one matrix of the type
$\diag\{\lambda,\lambda,\mu\}+e_{12}$ and $\lambda\mu\neq 0$.
Consider the algebraic expression $\mu \lambda^{-1}$.  If not
constant, then it takes on almost all values, so assume that it is a
constant $\delta$. Then $\delta\ne -2$, since otherwise this
matrix will be the unique matrix of trace 0 on the line $a+tb$ and
thus an element of $\widetilde {UD}$, contrary to
Lemmas~\ref{divAm} and \ref{div2}. Consider the polynomial
$q=p-\frac{\tr p}{\delta+2}$. At the same point $t$ it takes on the
value $\diag\{0,0,(\delta-1)\lambda\}+e_{12}$. Hence all three
pairwise distinct roots of the equation $\det q(x_1+t\tilde
x_1,x_2,\dots,x_m)=0$ will give us a matrix of the form
$\diag\{0,0,*\}+e_{12}$ (otherwise we have uniqueness and thus an
element of $\widetilde {UD}$), contradicting Lemma~\ref{div2}. Therefore $q$
satisfies an equation $\omega_2(q)=0$. Hence, $p$ satisfies an
equation $\omega_1(p)^2-c\omega_2(p)=0$, for some constant
 $c,$
a contradiction. Hence almost all non-diagonalizable matrices
belong to the image of $p$, and they are almost all matrices of
discriminant $0$ (a subvariety of $M_3(K)$ of codimension $1$). By
Amitsur's Theorem, $\Image p$ cannot be a subset of the
discriminant surface. Thus, $\Image p$ is dense in $M_3(K)$.
\end{proof}
\begin{rem}\label{rem_sl3}
Note that if $\omega_1(p)$ is identically zero, and $\omega_2 (p)$
is not identically zero, then $\Image p$  contains a matrix
similar to $\diag\{1,1,-2\}+e_{12}$. Hence $\Image p$ contains all
diagonalizable trace zero matrices (perhaps with the exception of
the diagonalizable matrices of discriminant $0$, i.e. matrices
similar to $\diag\{c,c,-2c\}$), all non-diagonalizable
non-nilpotent trace zero matrices, and all matrices $N$ for which
$N^2=0$. Nilpotent matrices of order $3$ also belong to the image
of $p$, as we shall see in Lemma \ref{V3inimage}.
\end{rem}

\section{Proof of the main Theorem}
\begin{lemma}\label{V3} A matrix is 3-scalar iff its eigenvalues are in   $\{\gamma, \gamma\varepsilon, \gamma\varepsilon^2: \gamma \in K\},$
where $\gamma^3\in K$ is its determinant. The variety $V_3$ of
3-scalar matrices has dimension 7.
\end{lemma}
\begin{proof} The first assertion is immediate since the
characteristic polynomial is $x^3 -\gamma^3.$ Hence ${V_3}$ is a
variety. The second assertion follows since the invertible
elements of $V_3$ are defined by two equations: $\tr (x)=0$ and
$\tr (x^{-1})=0$ and thus a $V_3$ is a variety of codimension~$2$.
\end{proof}
\begin{lemma}\label{V3inimage}
Assume $\Char K\neq 3$. If $p$ is neither PI nor central, then the
variety $V_3$ is contained in $\Image p$.
\end{lemma}
\begin{proof}
According to Lemma ~\ref{linear} there exist matrix units
$E_1,E_2,\dots,E_m$ such that $p(E_1,E_2,\dots,E_m)=e_{1,2}$.
Consider the mapping $\chi$ described in the proof of Theorem
\ref{dense}. For any triples $T_i=(t_{1,i},t_{2,i},t_{3,i})$, let
$$f(T_1,T_2,\dots,T_m)=p(\dots,t_{1,i}E_i+t_{2,i}\chi(E_i)+t_{3,i}\chi^2(E_i),\dots).$$
 $\Image f$ (a subset of $\Image p$) is a subset
of the  $3-$dimensional linear space
$$L=\{\alpha e_{12}+\beta e_{23}+\gamma e_{31},\ \alpha ,\beta ,\gamma\in
K\}.$$ Since $e_{12}$, $e_{23}$ and $e_{31}$ belong to $\Image f$,
we see that $\Image f$ is dense in $L$, and hence at least one
matrix $a = \alpha e_{12}+\beta e_{23}+\gamma e_{31}$ for
$\alpha\beta \gamma\neq 0$ belongs to $\Image p$. Note that this
matrix is $3-$central. Thus the variety ${V_3}$,  excluding the
nilpotent matrices, is contained in $\Image p$. The nilpotent
matrices of order $2$ also belong to the image of $p$ since they
are similar to $e_{12}$.

Let us show that all nilpotent matrices of order $3$ (i.e.,
matrices similar to $e_{12}+e_{23}$), also belong to $\Image p$.
We have the multilinear polynomial
$$f(T_1,T_2,\dots,T_m)=q(T_1,T_2,\dots,T_m)e_{12}+
r(T_1,T_2,\dots,T_m)e_{23}+s(T_1,T_2,\dots,T_m)e_{31},$$ therefore
$q,r$ and $s$ are three scalar multilinear polynomials. Assume
there is no nilpotent matrix of order $3$ in $\Image p$. Then we
have the following: if $q=0$ then either $rs=0$, if $r=0$ then
$sq=0$, and if $s=0$ then $qr=0.$ Assume $q_1$ is the greatest
common divisor of $q$ and $r$ and $q_2=\frac{q}{q_1}$. Note both
$q_i$ are multilinear polynomials defined on disjoint sets of
variables. If $q_1=0$ then $r=0$ and if $q_2=0$ then $s=0$. Note
there are no double efficients, and thus $r=q_1r'$ is a multiple
of $q_1$ and $s=q_2s'$ is a multiple of $q_2$. The polynomial $r'$
cannot have common devisors with $q_2$, therefore if we consider
any generic point $(T_1,\dots,T_m)$ on the surface $r'=0$ then
$r(T_1,\dots,T_m)=0$ and $q(T_1,\dots,T_m)\neq 0$. Hence
$s(T_1,\dots,T_m)=0$ for any generic $(T_1,\dots,T_m)$ from the
surface $r'=0.$ Therefore $r'$ is the divisor of $s$. Remind both
$q_1$ and $q_2$ are multilinear polynomials defined on disjoint
subsets of $\{T_1,T_2,\dots,T_m\}$. Without loss of generality
$q_1=q_1(T_1,\dots,T_k)$, and $q_2=q_2(T_{k+1},\dots,T_m)$.
Therefore $r'=r'(T_{k+1},\dots,T_m)$ and it is divisor of $s$.
Also remind $s=s'q_2$ so $q_2(T_{k+1},\dots,T_m)$ is also divisor
of $s$. Hence $r'=cq_2$ where $c$ is constant. Thus
$r=q_1r'=cq_1q_2=cq$. However there exist $(T_{k+1},\dots,T_m)$
such that $q=0$ and $r=1$ (i.e. such that
$f(T_{k+1},\dots,T_m)=e_{23}$). A contradiction.
\end{proof}
\begin{rem}
When $\Char K=3$,   then $V_3$ is the space of the matrices with
equal eigenvalues (including also scalar matrices). The same proof
shows that all nilpotent matrices belong to the image of $p$, as
well as all matrices similar to $cI+e_{12}+e_{23}$. But we do not
know how to show that scalar matrices and matrices similar to
$cI+e_{12}$ belong to the image of $p$.
\end{rem}
\begin{proof-of-thm}{\ref{main}}
First assume that $\Char K\ne 3$. According to Lemma
\ref{V3inimage} the variety ${V_3}$ is contained in $\Image p$.
Therefore $\Image p$ is either the set of $3$-scalar matrices, or
some $8-$dimensional variety (with $3$-scalar subvariety), or is
$9-$dimensional (and thus dense).

It remains  to classify the possible $8-$dimensional images. Let
us consider all matrices $p(E_1,\dots,E_m)$ where $E_i$ are matrix
units. If all such matrices have trace~0, then $\Image p$ is dense
in $\sl_3(K)$, by Theorem \ref{multi_tr=0_3}. Therefore we may
assume that at least one such matrix $a$ has eigenvalues $\alpha
,\ \beta $ and $\gamma$ such that $\alpha +\beta +\gamma\neq 0.$
By Theorem \ref{dense} we cannot have $\alpha +\beta +\gamma,\
\alpha +\beta \varepsilon+\gamma\varepsilon^2$ and $\alpha +\beta
\varepsilon^2+\gamma\varepsilon$ all nonzero. Hence $a$ either is
scalar, or a linear combination (with nonzero coefficients) of a
scalar matrix and $\diag\{1,\varepsilon,\epsilon^2\}$ (or with
$\diag\{1,\varepsilon^2,\epsilon\}$, without loss of generality -
with $\diag\{1,\varepsilon,\varepsilon^2\}$). By Theorem
\ref{equation}, if $\Image p$ is not dense, then $p$ satisfies an
equation of the type $(\tr(p))^2=\gamma\tr(p^2)$ for some
$\gamma\in K$. Therefore, if a scalar matrix belongs to $\Image
p$, then $\gamma=\frac{1}{3}$ and $\Image p$ is the set of
$3-$scalar plus scalar matrices. If the matrix $a$ is not scalar,
then it is a linear combination of a scalar matrix and
$\diag\{1,\varepsilon,\varepsilon^2\}$. Hence, by Remark
\ref{1-codim}, $\Image p$ is also the set of $3-$scalar plus
scalar matrices. In any case, we have shown that $\Image p$ is
either $\{0\}$, $K$, the set of $3-$scalar matrices, the set of
$3-$scalar plus scalar matrices (matrices with eigenvalues
$(\alpha +\beta ,\alpha +\beta \varepsilon,\alpha +\beta
\varepsilon^2)$), $\sl_3(K)$ (perhaps lacking nilpotent matrices
of order $3$), or is dense in $M_3(K)$.

If $\Char K=3$, then by Remark \ref{char3-classification} the multilinear
polynomial $p$ is either trace-vanishing or $\Image p$ is dense in
$M_3(K).$ If $p$ is trace-vanishing, then by Theorem~\ref{multi_tr=0_3},
$\Image p$ is one of the following:
\{0\}, the set of scalar matrices,
the set of $3-$scalar matrices, or
for each triple
$\lambda_1+\lambda_2+\lambda_3=0$ there exists a matrix $M\in\Image
p$ with eigenvalues $\lambda_1,\ \lambda_2$ and $\lambda_3$.
\end{proof-of-thm}
\section{Open problems}

\begin{prm}\label{3s}
Does there actually exist a multilinear polynomial whose  image
evaluated on $3\times3$ matrices consists of $3-$scalar matrices?
\end{prm}

\begin{prm}\label{s3s}
Does there actually exist a multilinear polynomial whose  image
evaluated on $3\times3$ matrices   is the set of scalars plus
$3-$scalar matrices?
\end{prm}

\begin{rem}
Problems \ref{3s} and \ref{s3s}  both have the same answer.   If
they both have affirmative answers,   such a polynomial would a
counter-example to   Kaplansky's problem.
\end{rem}

    \begin{prm}
Is it possible that the image of a multilinear polynomial
evaluated on $3\times3$ matrices is dense but not all of $M_3(K)$?
\end{prm}
 \begin{prm}
Is it possible that the image of a multilinear polynomial
evaluated on $3\times3$ matrices is the set of all trace-vanishing matrices without discriminant vanishing diagonalazable matrices?
\end{prm}

%

\end{document}